\documentclass{article}
\usepackage{graphicx}
\usepackage{amsmath}
\usepackage{amsfonts}
\usepackage{amssymb}
\usepackage{amsthm}
\usepackage{marvosym}
\usepackage{xcolor}
\usepackage{comment}
\usepackage[T1]{fontenc}
\usepackage[utf8]{inputenc}
\usepackage[english]{babel}
\usepackage{tikz-cd}
\newtheorem{theorem}{Theorem}[section]
\newtheorem{proposition}[theorem]{Proposition}
\newtheorem{corollary}[theorem]{Corollary}
\newtheorem{lemma}[theorem]{Lemma}
\newtheorem{definition}[theorem]{Definition}
\newtheorem{remark}[theorem]{Remark}
\newtheorem{conjecture}[theorem]{Conjecture}
\usepackage{fullpage}

\newcommand{\wa}{A\rtimes W}
\newcommand{\WX}{W\langle X\rangle}
\newcommand{\FX}{F\langle X\rangle}
\newcommand{\WYZ}{W\langle Y\cup Z\rangle}
\newcommand{\bN}{\mathbb{N}}
\newcommand{\cS}{\mathcal{S}}
\newcommand{\cW}{\mathcal{W}}
\newcommand{\cV}{\mathcal{V}}

\DeclareMathOperator{\diag}{diag}

\DeclareMathOperator{\spn}{Span}
\DeclareMathOperator{\gpi}{\textnormal{gPI}}
\DeclareMathOperator{\Polid}{\textnormal{PI}}
\DeclareMathOperator{\gid}{\textnormal{gId}}

\DeclareMathOperator{\id}{\textnormal{Id}}

\DeclareMathOperator{\Hilb}{gHilb}
\DeclareMathOperator{\Sym}{Sym}
\DeclareMathOperator{\mult}{mult}
\DeclareMathOperator{\var}{var}
\DeclareMathOperator{\F}{gF}
\DeclareMathOperator{\img}{Img}

\title{Cocharacters of generalized polynomial identities}

\author{
Sebastiano Argenti \footnote{Corresponding author}  \footnote{Dipartimento di Scienze di Base e Applicate, Universit\`{a} degli Studi della Basilicata, Viale dell'Ateneo Lucano 10, 85100, Potenza (Italy) \\ {\em E-mail address}: sebastiano.argenti@unibas.it} , 
Giovanni Busalacchi \footnote{Dipartimento di Matematica e Informatica, Universit\`{a} degli Studi di Palermo, via Archirafi 34, 90123, Palermo (Italy) \\ {\em E-mail address}: giovanni.busalacchi@unipa.it. } \footnote{ G. Busalacchi was supported by
Doctoral School on Mathematics and Computational Sciences, Universities of Messina, Catania and Palermo. \\
Both authors were partially supported by GNSAGA of INDAM.}
}
\date{}

\begin{document}

\maketitle

\begin{abstract}
In this paper we extend the cocharacter theory to generalized identities of $W$-algebras. We prove that the Hilbert series of the relatively free $W$-algebra admits an expansion in terms of Schur functions whose coefficients coincide with generalized cocharacter multiplicities. Moreover, we prove analogues of the Hook and Strip theorems for $W$-algebras and we derive growth bounds for generalized codimension and colenght sequences. Finally, we establish that every variety $\mathcal{V}$ of $W$-algebras is generated by the Grassmann envelope of a finitely generated $W$-superalgebra, and if $\mathcal{V}$ satisfies a generalized Capelli set, then it is generated by a finitely generated $W$-algebra.\\
    
    
\end{abstract}

\section{Introduction}
In the framework of associative algebras endowed with additional structure, the theory of $W$-algebras and their generalized polynomial identities was initiated by Amitsur \cite{AM1} in 1965. Then, this subject was further developed by several authors and we refer to \cite{BMM96} and its bibliography for an extensive treatment of the topic. These identities extend the classical notion of polynomial identities by allowing the elements of a fixed algebra $W$ to appear between the variables. A combinatorial approach to generalized identities has gained increasing attention in recent years. In fact, \cite{Gor10} proved the existence of the generalized exponent of any finite dimensional algebra $A$ in case $A=W$ acts on itself. This result was later extended in \cite{MR25b} to any unital finite dimensional $W$-algebra $A$, where $W$ is not necessarily equal to $A$.

A central object of study in this context is the free $W$-algebra $W\langle X \rangle$, in which we can find the ideal of generalized identities $\gid (A)$ containing all polynomial relations satisfied by a given $W$-algebra $A$. The structure of this ideal has been extensively investigated, notably for specific classes of $W$-algebras, such as full matrix algebras \cite{BS15} or algebras of triangular matrices \cite{MR25a}.
Parallel developments in \cite{MR25b} led to a better understanding of the role of multipliers in defining bimodule actions and its connection to the the growth of $T_W$-ideals.

One of the most powerful tools in the theory of ordinary PI-algebras is the use of representation theory of symmetric groups and general linear groups, via the analysis of cocharacters and codimension sequences \cite{Ber82}. Classical results have shown that the cocharacter multiplicities of PI-algebras lie within a hook \cite{AR82}, and, under suitable finiteness conditions, they lie within a strip \cite{Reg79}. This facts have been used to obtain estimates of the asymptotic behavior of the codimension sequence.

In this paper, we further extend these ideas to the setting of generalized identities by proving the analogous of the hook and strip theorems for $W$-algebras. More precisely, we study the actions of the symmetric group $S_n$ on multilinear generalized polynomials and of the general linear group $GL_k(F)$ on homogeneous generalized polynomials. Furthermore, we consider the Hilbert series of relatively free $W$-algebras and we prove that they admit an expansion in terms of Schur functions whose coefficients coincide with generalized cocharacter multiplicities (Theorem \ref{th: Hilbert}). As a consequence, for any $W$-algebra $A$ satisfying an ordinary polynomial identity, we derive growth bounds for generalized codimension and colength sequences and we show that they behave similarly to their ordinary counterparts (Theorems \ref{th: hook}).
Moreover, if $A$ has unity, we can prove better bounds on its generalized multiplicities (Lemma \ref{lem: multiplicity bound}) and, in this case, the hook determined by the generalized cocharacters coincides with the hook determined by the ordinary cocharacters (Theorem \ref{th: good_hook}).

Finally, we show that every variety $\cV$ of $W$-algebras is generated by the Grassmann envelope of a finitely generated $W$-superalgebra (Theorem \ref{th: Grassmann envelope}). If $\cV$ satisfies a generalized Capelli set, then it is generated by a finitely generated $W$-algebra (Corollary \ref{cor: fin_gen}). 

\section{Preliminaries}

Let $F$ be a field of characteristic zero. Throughout this paper all the algebras will be associative and defined over $F$.

Let $W$ be a finite dimensional associative algebra with unity. We say that the associative algebra $A$ is a $W$-algebra if $A$ is an unitary $W$-bimodule such that, for any $w\in W$ and $a_1,a_2\in A$, we have \begin{equation}\label{eq: W axioms}
    w(a_1a_2)=(wa_1)a_2,\ (a_1a_2)w= a_1(a_2w),\ (a_1w)a_2=a_1(wa_2)
\end{equation}

Given the $W$-algebras $A$ and $B$, an homomorphism of $W$-algebras from $A$ to $B$ is an algebra homomorphism $\varphi:A\rightarrow B$ such that $\varphi(w_1aw_2)=w_1\varphi(a)w_2$, for each $a\in A$ and $w_1,w_2\in W$. Thus, for a fixed $W$, the class of $W$-algebras is a category. It is also a variety in the sense of universal algebra. 

The variety of associative $W$-algebras contains the free associative $W$-algebra $\WX$, generated by the countable set of variables $X=\{x_1,x_2,\dots\}$ which is defined, up to isomorphism, by the following universal property: for every $W$-algebra $A$ and every function $f:X\rightarrow A$, there exists a unique $W$-algebra homomorphism $\tilde{f}: \WX \rightarrow A$ extending $f$, that is $\tilde{f}\mid_X = f$.

In \cite{MR25a}, the authors give a combinatorial construction of $\WX$, depending on a choice of basis of $W$. Here, we give an alternative, basis independent, construction of $\WX$. Consider the vector space $V_X$ consisting of the finite formal linear combinations of the variables in $X$, and, for each $n\geq 1$, define \begin{equation}\label{eq: W^n}
	W^{(n)}\langle X\rangle = W\otimes \underbrace{V_X\otimes W \otimes V_X \otimes \cdots W \otimes V_X}_n \otimes W 
\end{equation}
For each $n,m\geq1$, there exists a natural map $W^{(m)}\langle X\rangle\otimes W^{(n)}\langle X\rangle \rightarrow W^{(m+n)}\langle X\rangle$ obtained by contracting the two neighboring $W$ factors in the tensor product according to the multiplication map $W\otimes W\rightarrow W$ of $W$. This induces an algebra structure on \begin{equation}\label{eq: WX}
	W\langle X\rangle = \bigoplus_{n=1}^\infty W^{(n)}\langle X\rangle
\end{equation}
Moreover, $W\langle X \rangle$ is naturally a $W$-algebra using the multiplication of $W$ on the left and on the right, and it satisfies the universal property of a free $W$-algebra. Note that, once we fix a basis $w_1,\dots,w_d$ of $W$, a basis of $W\langle X\rangle$ is given by the monomials $w_{i_0}x_{j_1}w_{i_1}x_{j_2}\cdots w_{i_{n-1}}x_{j_n}w_{i_n}$, with $i_0,\dots,i_n\in\{1,\dots,d\}$ and $j_1,\dots,j_n\in \mathbb{N}$, which is precisely the basis given in \cite{MR25a}.

An element $f=f(x_1,\ldots,x_n)$ in $W \langle X \rangle$ is called $W$-generalized polynomial or simply generalized polynomial (if the role of $W$ is clear).

\begin{remark}\label{rem: constant term}
Note that we consider polynomials without constant terms. In fact, in Equation \eqref{eq: WX} the summation index begins at $n=1$.
\end{remark}

Let $A$ be a $W$-algebra and let $f=f(x_1,\ldots,x_n) \in W\langle X \rangle$ be a generalized polynomial. We say that $f$ is a generalized $W$-identity, or simply generalized identity, of $A$ if $f(a_1,\ldots,a_n)=0$ for any $a_1,\ldots,a_n \in A$. Equivalently, $f\in W\langle X \rangle$ is an identity if and only if it is in the kernel of every $W$-algebra homomorphism $\WX\rightarrow A$. We say that a $W$-algebra $A$ is a $\gpi$-algebra if satisfies a non trivial generalized identity. The set of all generalized identities of $A$ is denoted by $\gid(A)$ and it is a $T_W$-ideal, that is, an ideal of $W\langle X \rangle$ which is invariant under the $W$-bimodule action and under all  $W$-algebra endomorphisms of $W\langle X \rangle$. Given a non-empty set $S \subseteq W \langle X \rangle$, the class of all $W$-algebra $A$ such that $f\in\gid(A)$ for all $f \in S$ is called the variety $\mathcal{V}=\var^W(S)$ determined by $S$. The variety $\cV=\var^W(A)$ generated by the $W$-algebra $A$ is the variety determined by the generalized identities of $A$.

Since we assume that $W$ has a unit $1_W$ acting as the identity on $A$, then the ordinary identities of $A$ are contained in its generalized identities. In fact, the ordinary polynomials of the associative free algebra $\FX$ can be also considered generalized polynomials, using the identification of monomials: \[x_{i_1}x_{i_2}\cdots x_{i_n} = 1_Wx_{i_1}1_Wx_{i_2}1_W\cdots 1_Wx_{i_n}1_W\]
In this way, $\FX\subseteq \WX$ and $\id(A)\subseteq \gid(A)$.

We recall some invariants associated to any $T_W$-ideal, following the work of \cite{MR25a}.
For any $n\in \bN$, the space $gP_n=gP_n^W\subseteq\WX$ of multilinear generalized polynomials in the variables $x_1,\dots,x_n$ is defined as
\[ gP_n^W =\spn\{ w_1x_{\sigma(1)}w_2x_{\sigma(2)}w_3\cdots w_nx_{\sigma(n)}w_{n+1} \mid \sigma\in S_n, w_1,\dots,w_{n+1} \in W \}  \]
As in the ordinary case, since we assume that $F$ has characteristic zero, every $T_W$-ideal is completely determined by its multilinear generalized polynomials, that is $T_W$ is generated by $\bigcup_{n\in\bN} (T_W\cap gP_n)$. For a fixed $W$-algebra $A$ and for any $n\in\bN$, we define the space
\[ gP_n(A)=\frac{gP_n}{gP_n\cap \gid(A)}. \]
The $n$-th generalized (multilinear) codimension is $gc_n(A)=\dim gP_n(A)$. Note that $gc_n(A)$ is finite since $W$ finite dimensional. 

The symmetric group $S_n$ acts on $gP_n$ by substitution of the variables: any $\sigma\in S_n$ acts on $f(x_1,\dots,x_n)\in gP_n$ as
\[ \sigma \ast f(x_1,\dots,x_n) = f(x_{\sigma(1)},\dots,x_{\sigma(n)}) \]
The space $gP_n\cap \gid(A)$ is invariant under the $S_n$-action, so $gP_n(A)$ is an $S_n$-module, too. Recall that the finite dimensional irreducible modules of $S_n$ are the Shur modules $\cS_\lambda$, parameterized by the partitions $\lambda$ of $n$, that is by finite sequence of integers $\lambda=(\lambda_1,\ldots,\lambda_r)$ such that $\lambda_1 \geq \cdots \geq \lambda_r > 0$ and $\sum_{i=1}^r=n$. In this case we write $\lambda \vdash n$ or $|\lambda|=n$. Denote with $\chi_\lambda$ the character of $\cS_\lambda$.
This allows to define the $n$-th generalized cocharacter and the corresponding multiplicities as 
\begin{equation}\label{eq: gen_cocharacter}
	g\chi_n(A) = \sum_{\lambda\vdash n} m_\lambda \chi_\lambda
\end{equation}
Clearly, computing the cocharacter $g\chi_n(A)$ in the unit element of the symmetric group $S_n$, that is the identity permutation, we have that the $n$-th generalized multilinear codimension of $A$ corresponds to \begin{equation}\label{eq: gen_cod_formula}
	gc_n(A) =  \sum_{\lambda \vdash n}m_\lambda d_\lambda 
\end{equation}
where $d_\lambda= \dim \cS_\lambda$ is the dimension of the Shur module $\cS_\lambda$. We also define the $n$-th generalized colength as
\[ gl_n(A) = \sum_{\lambda \vdash n} m_\lambda \]

\section{Hilbert series of the relatively free $W$-algebra}

In this section we introduce the relatively free $W$-algebra and the corresponding Hilbert series. As in the classical case, we show its relationship with the generalized cocharacters.

Let $W\langle x_1,\dots,x_k\rangle\subseteq\WX$ be the $W$-subalgebra of $\WX$ generated by $x_1,\dots,x_k$. If $A$ is a $W$-algebra, the relatively free $W$-algebra of $\var^W(A)$ of rank $k\in \bN$ is 
\[ \F_k(A)=\frac{W\langle x_1,\dots,x_k\rangle}{W\langle x_1,\dots,x_k\rangle\cap \gid(A)} \]

The free algebra $W\langle x_1,\dots,x_k\rangle$ of rank $k$ has a natural multigrading obtained by counting the degree of each variable $x_1,\dots,x_k$ separately. More precisely, the generalized polynomial $f\in W\langle x_1,\dots,x_k\rangle$ has homogeneous multidegree $\underline{\alpha}=(\alpha_1,\dots,\alpha_k)\in \bN^k$ if the variable $x_i$ occurs $\alpha_i$ times in each monomial of $f$. Since the ideal $W\langle x_1,\dots,x_k\rangle\cap \gid(A)$ is multihomogeneous, the relatively free algebra $\F_k(A)$ inherits the multigrading. For each $\underline{\alpha}\in \bN^k$, let $V_k^{(\underline{\alpha})}\in W\langle x_1,\dots,x_k\rangle$ be the space of homogeneous generalized polynomials of multidegree $\underline{\alpha}$, and define
\[ V_k^{(\underline{\alpha})}(A)= \frac{V_k^{(\underline{\alpha})}}{V_k^{(\underline{\alpha})}\cap \gid(A)}\]
Note that $V_k^{(1,\dots,1)}(A)=gP_k(A)$. The $k$-th Hilbert series of $\F_k(A)$ is the formal power series:
\[ \Hilb(A; t_1,\dots,t_k) = \sum_{\alpha_1,\dots,\alpha_k\in\bN} \dim V_k^{(\underline{\alpha})}(A)\  t_1^{\alpha_1} t_2^{\alpha_2}  \cdots t_k^{\alpha_k} \]

The space of homogeneous generalized polynomials of degree $n$ in $k$ variables is 
\[ W_k^{(n)} = W^{(n)}\langle X\rangle\cap W\langle x_1,\dots,x_k\rangle = \bigoplus_{\substack{\alpha_1,\dots,\alpha_k\in\bN \\ \alpha_1+\dots+\alpha_k = n}} V_k^{(\underline{\alpha})} \]

For a fixed $W$-algebra $A$ and any $k,n\in\bN$, define
\[ W_k^{(n)}(A) = \frac{W_k^{(n)}}{W_k^{(n)}\cap \gid(A)} \]
Then, the $n$-th generalized homogeneous codimension is $gC_k^n (A)= \dim  W_k^{(n)}(A)$.

The general linear group $GL_k(F)$ acts by automorphisms on $W\langle x_1,\dots, x_k\rangle$. In fact, each $u=(u_{ij})_{i,j\leq k}\in GL_k(F)$ naturally acts on $\spn\{x_1,\dots,x_k\}$ according to
\[ u \ast x_j = \sum_{i=1}^k u_{ij}x_i  \]
and this action is naturally extended to each $f(x_1,\dots,x_k)\in W\langle x_1,\dots, x_k\rangle$ according to
\[ u \ast f(x_1,\dots,x_k) = f(u\ast x_1,\dots, u\ast x_k) \]
Each homogeneous space $W^{(n)}_k$ is stable under the $GL_k(F)$-action and the same holds for $W^{(n)}_k\cap \gid(A)$, so $W^{(n)}_k(A)$ is a $GL_k(F)$-module and we denote its character with $\phi_k^n(A)$. 

Recall, that the character $\chi_V$ of a finite dimensional polynomial representation $V$ of $GL_k(F)$ is an invariant function, that is there exists a polynomial $p(t_1,\dots,t_k)\in F[t_1,\dots,t_k]$ such that the character computed in $u\in GL_k(F)$ is $\chi_V(u)=p(\lambda_1\dots,\lambda_k)$, where $\lambda_1\dots,\lambda_k$ are the eigenvalues of $u$ counted with multiplicities. Moreover, $p$ is a symmetric polynomial, that is, it is invariant under permutation of its variables. The space of symmetric polynomials in $k$ variables is denoted with $\Sym^k$ and it is a subalgebra of $F[t_1,\dots,t_k]$. Therefore, with abuse  of language, we identify the character of a finite dimensional representations of $GL_k(F)$ with the corresponding symmetric polynomial in $\Sym^k$.
The finite dimensional irreducible representations of $GL_k(F)$ are parameterized by the partitions $\lambda$ of height $h(\lambda)$ at most $k$. To each partition $\lambda$ with $h(\lambda)\leq k$ corresponds the Weyl module $\cW^k_\lambda$ whose character is the Shur polynomial $s^{k}_\lambda(t_1,\dots,t_k)$.

Moreover, recall that a polynomial representation $V$ of $GL_k(F)$ has homogeneous degree $n\in\bN$ if for each scalar matrix $\alpha I_k \in GL_k(F)$, $\alpha \in F^\times$, and each $v\in V$ we have $\alpha I_k\ast v = \alpha^n v$. Every submodule of a $GL_k(F)$-module of homogeneous degree $n$ has homogeneous degree $n$, too. Note that $\cW^k_\lambda$ has homogeneous degree $|\lambda|$ and that $W^{(n)}_k(A)$ has homogeneous degree $n$. Therefore, we can write $\phi_k^n(A)$, i.e. the $n$-th generalized homogeneous cocharacter in $k$ variables, as \begin{equation}\label{eq: gen_coch_temp}
	\phi_k^n(A) = \sum_{\substack{\lambda\vdash n \\ h(\lambda)\leq k}} m'_{\lambda,k} s^k_\lambda(t_1,\dots,t_k)
\end{equation}
for some multiplicities $m'_{\lambda,k}\in F$. In the following, we prove that the multiplicities we just defined are actually equal to the multiplicities $m_\lambda$ defined in Equation \eqref{eq: gen_cocharacter}. 

We first need to recall an explicit construction of the irreducible representations of the symmetric group and of the general linear group. Let $\lambda\vdash n$ and consider a Young Tableaux $T_\lambda$ of shape $\lambda$, that is a filling of the Young diagram of shape $\lambda$ with the numbers $1,\dots,n$. The element $e_{T_\lambda}\in F[S_n]$ is defined as \begin{equation}\label{eq: e_T}
    e_{T_\lambda}=\sum_{\substack{\sigma \in R_{T_\lambda} \\ \tau \in C_{T_\lambda}}} (-1)^\tau \sigma \tau 
\end{equation}
where $R_{T_\lambda}$ and $C_{T_\lambda}$ denote the row-stabilizer and the column-stabilizer of $T_\lambda$, respectively. Then $F[S_n]e_{T_\lambda}$ is a minimal left ideal in $F[S_n]$ and is isomorphic to the Shur module $\cS_\lambda$.

Let $U$ be a $k$-dimensional vector space on which $GL_k(F)$ acts naturally. Then, $GL_k(F)$ acts on left on $U^{\otimes n}$ according to 
\[u\ast (v_1\otimes\cdots\otimes v_n)= (u*v_1)\otimes\cdots\otimes(u*v_n)\]
Moreover, $S_n$ acts on the right on $U^{\otimes n}$ by permuting places, that is 
\[(v_1\otimes\cdots\otimes v_n) \ast \sigma = v_{\sigma(1)}\otimes\cdots\otimes v_{\sigma(n)}\]
The two actions centralize each other. The Weyl module $\cW_\lambda^k$ is isomorphic to $U^{\otimes n}*e_{T_\lambda}$, for any Young Tableaux $T_\lambda$ of shape $\lambda\vdash n$.

We can now prove a Lemma that generalizes the techniques used in \cite[Lemma 7]{Ber96}. For any $GL_k(F)$-module $V$ of homogeneous degree $n\leq k$, we define the multilinear part $V^{\mult}$ as
\[ V^{\mult} = \{ v\in V\mid \diag(\alpha_1,\dots,\alpha_n,1,\dots,1)\ast v = \alpha_1\cdots\alpha_n v\quad \forall \alpha_1,\dots,\alpha_k\in F^\times \} \]
We embed $S_n$ in $GL_k(F)$ using the permutation matrices acting on the first $n$ rows. Then, $V^{\mult}$ is stable under the action of $S_n$, so it is a $S_n$-module.

\begin{lemma}\label{lem: mult}
	Let $k\geq n$. Then, \begin{itemize}
		\item $(V\oplus W)^{\mult} \cong_{S_n} V^{\mult} \oplus W^{\mult}$, for any $GL_k(F)$-modules $V$ and $W$ of homogeneous degree $n$.
		\item $( \cW_\lambda^k )^{\mult} \cong  \cS_\lambda$ as $S_n$-modules, for any $\lambda\vdash n$.
	\end{itemize}
\end{lemma}
\begin{proof}
	The first part of the statement is clear. For the second part, consider a $k$-dimensional vector space $U=\spn\{e_1,\dots,e_k\}$ and its tensor power $U^{\otimes n}$ with left $GL_k(F)$-action and right $S_n$-action, as above.
	Note that the multilinear part of $U^{\otimes n}$ is given by
	\[ (U^{\otimes n})^{\mult} = \spn\{ e_{\sigma(1)}\otimes\cdots\otimes e_{\sigma(n)} \mid \sigma \in S_n  \} \]
	Therefore, $(U^{\otimes n})^{\mult} \cong F[S_n]$ as $S_n$ bimodule. Let $\lambda\vdash n$ and consider a Young Tableaux $T_\lambda$ of shape $\lambda$ and the corresponding minimal idempotent $e_{T_\lambda} \in F[S_n]$.
	We have the following isomorphisms of left $S_n$ modules:
	\[( \cW_\lambda^k )^{\mult} \cong  (U^{\otimes n}*e_{T_\lambda})^{\mult} \cong  (U^{\otimes n})^{\mult}*e_{T_\lambda}  \cong  F[S_n]e_{T_\lambda} \cong \cS_\lambda\] 
\end{proof}

The next theorem is a straightforward generalization of the arguments of \cite{Ber82} concerning ordinary polynomial identities.

\begin{theorem}\label{th: Hilbert}
	Let $A$ be a $W$-algebra with generalized cocharacter sequence \begin{equation}\label{eq: gen_coch}
		g\chi_n(A) = \sum_{\lambda\vdash n} m_\lambda \chi_\lambda
	\end{equation} 
	Then, for each $k\in\bN$, the generalized homogeneous cocharacters sequence of $A$ in $k$ variables corresponds to 
	\begin{equation}\label{eq: hom_gen_coch}
		\phi_k^n(A) = \sum_{\substack{\lambda\vdash n \\ h(\lambda)\leq k}} m_{\lambda} s^k_\lambda(t_1,\dots,t_k)
	\end{equation}
	Moreover, the Hilbert series of $F_k(A)$ in corresponds to 
	\begin{equation}\label{eq: Hilb_symmetric}
		\Hilb(A;t_1,\dots,t_k) = \sum_{n=1}^\infty \phi_k^n(A) = \sum_{h(\lambda)\leq k} m_\lambda s^k_\lambda(t_1,\dots,t_k)
	\end{equation}
\end{theorem}
\begin{proof}
 	
	Let $k\geq n$ and let $m'_{\lambda,k}$, for $\lambda\vdash n$, $ h(\lambda)\leq k$, be the multiplicities of the $n$-th generalized cocharacter of $A$ as in Equation \eqref{eq: gen_coch_temp}.
	By Lemma \ref{lem: mult},
	\[  (W_k^{(n)}(A))^{\mult } \cong \sum_{\substack{\lambda\vdash n \\ h(\lambda)\leq k}} m'_{\lambda,k} (\cW_\lambda^k)^{\mult} \cong \sum_{\lambda\vdash n } m'_{\lambda,k} \cS_\lambda \]
	On the other hand, 
	\[ (W_k^{(n)}(A))^{\mult } = P_n(A) \cong \sum_{\lambda\vdash n} m_\lambda \cS_\lambda \]
	This proves that $m'_{\lambda,k} = m_\lambda$ for each $\lambda\vdash n$ and $k\geq n$. So, Equation \eqref{eq: hom_gen_coch} is proved.
	
	Now, let $D=\diag(t_1,\dots,t_k)\in GL_k(F)$ be the diagonal matrix with coefficients $t_1,\dots,t_k\in F^\times$. The trace of the action of $D$ on $W_k^{(n)}(A)$ is, on one hand, 
	\[ \sum_{\substack{\alpha_1,\dots,\alpha_k\in\bN \\ \alpha_1+\dots+\alpha_k = n}} \dim V^{(\underline{\alpha})}(A)\  t_1^{\alpha_1} t_2^{\alpha_2}  \cdots t_k^{\alpha_k} \]
	and, on the other hand, corresponds to the cocharacter $\phi_k^n(A)$ computed in $D$. This proves Equation \eqref{eq: Hilb_symmetric}.
\end{proof}

The previous theorem shows that the Hilbert series $\Hilb(A;t_1,\dots,t_k)$ of $\F_k(A)$ is symmetric in the variables $t_1,\dots,t_k$. It also gives the analogue of Equation \eqref{eq: gen_cod_formula} for generalized homogeneous codimensions:

\begin{corollary}
	The generalized homogeneous codimensions sequence of $A$ in $k$ variables corresponds to
	\begin{equation}\label{eq: gen_hom_cod_formula}
		gC_k^n(A) = \sum_{\substack{\lambda\vdash n \\ h(\lambda)\leq k}} m_{\lambda} \dim \cW_\lambda^k  = \sum_{\substack{\lambda\vdash n \\ h(\lambda)\leq k}} m_{\lambda} s^k_\lambda(1,\dots,1) 
	\end{equation}
\end{corollary}
\begin{proof}
	It is enough to use Equation \eqref{eq: hom_gen_coch} to compute the cocharacter $\phi_k^n(A)$ in the unit element of $GL_k(F)$, that is the identity matrix.
\end{proof}

We end this section by giving some examples of generalized Hilbert series. The first example is the free $W$-algebra itself. Note that, in this case, $\F_k(\WX)=W\langle x_1,\dots,x_k\rangle$, since $\gid(\WX)=0$.

\begin{theorem}
	Let $W$ be a unitary associative algebra of dimension $d$. Then, the Hilbert series of the free $W$-algebra of rank $k$ is \begin{gather*}
		\Hilb(\WX; t_1,\dots,t_k) = \frac{d^2 (t_1+\ldots+t_k)}{1-d (t_1+\ldots +t_k)}
	\end{gather*}
\end{theorem}
\begin{proof}
	The standard $GL_k(F)$-module $V_k=\spn\{x_1,\dots,x_k\}$ has character $s^k_{(1)}= t_1+\dots + t_k$.
	By Equation \eqref{eq: W^n}, the space $W_k^{(n)}$ of degree $n$ homogeneous polynomials in $k$ variables is isomorphic, as a $GL_k(F)$-module, to the product of $n$ copies of $V_k$ and to $n+1$ copies of $W$, on which $GL_k(F)$ acts trivially. So the character of $W_k^{(n)}$ is $d^{n+1}(t_1+\dots + t_k)^n$ for any $k,n\geq 1$. Therefore,
	\[ \Hilb(\WX; t_1,\dots,t_k) = d\sum_{n=1}^{\infty} d^{n}(t_1+\dots + t_k)^n = \frac{d^2 (t_1+\ldots+t_k)}{1-d (t_1+\ldots +t_k)} \]
\end{proof}

In \cite{MR25a}, the authors computed the generalized cocharacter sequences of $UT_2$ with three different $W$-actions. We use their results to give the corresponding Hilbert series. We first recall the following result of \cite{Reg94} regarding the Young derived series:

\begin{proposition}[\cite{Reg94}]\label{prop: young}
	Assume that
	\[ \left(\prod_{i=1}^{k}\frac{1}{(1-t_i)}\right) \sum_{n=0}^{\infty}\sum_{\mu\vdash n} \alpha_\mu s^k_\mu = \sum_{n=0}^{\infty}\sum_{\lambda\vdash n} m_\lambda s^k_\lambda  \]
	Then, for each partition $\lambda=(\lambda_1,\dots,\lambda_k)$, we have that
	\[ m_\lambda = \sum_\mu \alpha_\mu \]
	where the sum runs trough all the partitions $\mu=(\mu_1,\dots,\mu_k)$ such that 
	\[ \lambda_{1} \geq \mu_1 \geq \lambda_{2} \geq \mu_2 \geq\dots\geq \lambda_{k} \geq \mu_k \]
\end{proposition}

We now consider the $W$-action of $UT_2$ on itself, that is $W=UT_2$ acting by left and right multiplication. As in \cite{MR25a}, we denote this case by $UT_2$ without any superscript.

\begin{theorem}
	The Hilbert series of the relatively free algebra of rank $k$ of $UT_2$ acting on itself is \begin{gather*}
		\Hilb(UT_2; t_1,\dots,t_k) = -3 +\frac{2}{(1-t_1)\cdots (1-t_k)} + \frac{t_1+\ldots +t_k +1}{(1-t_1)^2\cdots (1-t_k)^2}
	\end{gather*}
\end{theorem}
\begin{proof}
	By \cite[Theorem 4.4]{MR25a}, the generalized cocharacter sequence $g\chi_n(UT_2)=\sum_\lambda m_\lambda \chi_\lambda$ is given by
	\[ m_\lambda = \begin{cases}
		2n+3 &\text{if } \lambda = (n),\ n\geq 1 \\
		3(q+1) &\text{if } \lambda = (p+q,p),\ p\geq1, q\geq 0, \\
		q+1 &\text{if } \lambda = (p+q,p,1),\ p\geq1, q\geq 0, \\
		0 &\text{in all other cases}
	\end{cases} \]
	Therefore, the Hilbert series of $UT_2$ is 
	\[ \Hilb(UT_2; t_1,\dots,t_k) = \sum_{n\geq 1} (2n+3) s^k_{(n)} + \sum_{\substack{p\geq 1 \\ q\geq 0}} 3(q+1) s^k_{(p+q,p)} + \sum_{\substack{p\geq 1 \\ q\geq 0}} (q+1) s^k_{(p+q,p,1)}  \]
	We rewrite it as \begin{align*}
		\Hilb(UT_2; t_1,\dots,t_k) = &\left( \sum_{n\geq 0} s^k_{(n)} \right) +  2 \left( \sum_{n\geq 0} (n+1) s^k_{(n)} + \sum_{\substack{p\geq 1 \\ q\geq 0}} (q+1) s^k_{(p+q,p)} \right) + \\
		&\left( \sum_{\substack{p\geq 1 \\ q\geq 0}} (q+1) s^k_{(p+q,p)} + \sum_{\substack{p\geq 1 \\ q\geq 0}} (q+1) s^k_{(p+q,p,1)} \right) - 3
	\end{align*}
	Then, we use Proposition \ref{prop: young} to obtain
	\[ \Hilb(UT_2; t_1,\dots,t_k) = \left(\prod_{i=1}^{\infty}\frac{1}{(1-t_i)}\right) \left( 1 + 2\sum_{n\geq 0} s^k_{(n)} + \sum_{n\geq 1} s^k_{(n,1)} \right) -3  \]
	Now, note that applying \ref{prop: young} to the constant symmetric function of value 1, we obtain:
	\[\sum_{n\geq 0} s^k_{(n)} = \prod_{i=1}^{k}\frac{1}{(1-t_i)}\]
	Moreover, by \cite[Lemma 7]{BD12} the following equality holds:
	\[ \sum_{n\geq 1} s^k_{(n,1)} = 1 + (s^k_{(1)}-1)\prod_{i=1}^{k}\frac{1}{(1-t_i)} \]
	Therefore, we have
	\[ \Hilb(UT_2; t_1,\dots,t_k) = \left(\prod_{i=1}^{k}\frac{1}{(1-t_i)}\right) \left( 1 + 2\prod_{i=1}^{k}\frac{1}{(1-t_i)} + 1 + (s^k_{(1)}-1)\prod_{i=1}^{k}\frac{1}{(1-t_i)} \right) -3\]
	and using simple algebraic manipulations we obtain the desired result.
\end{proof}

We now consider the case of $UT_2$ with the action of $W=D=\spn\{e_{11},e_{22}\}\subseteq UT_2$ by left and right multiplication and we denote this case with $UT_2^D$.

\begin{theorem}
	The Hilbert series of the relatively free algebra of rank $k$ of $UT_2^D$ is \begin{gather*}
		\Hilb(UT_2^D; t_1,\dots,t_k) = -2 +\frac{2}{(1-t_1)\cdots (1-t_k)} + \frac{t_1+\ldots +t_k }{(1-t_1)^2\cdots (1-t_k)^2}
	\end{gather*}
\end{theorem}
\begin{proof}
	By \cite[Theorem 5.2]{MR25a}, the generalized cocharacter sequence $g\chi_n(UT_2^D)=\sum_\lambda m_\lambda \chi_\lambda$ is given by
	\[ m_\lambda = \begin{cases}
		n+2 &\text{if } \lambda = (n),\ n\geq 1 \\
		2(q+1) &\text{if } \lambda = (p+q,p),\ p\geq1, q\geq 0, \\
		q+1 &\text{if } \lambda = (p+q,p,1),\ p\geq1, q\geq 0, \\
		0 &\text{in all other cases}
	\end{cases} \]
	Then, proceeding as in the previous theorem we obtain the result.
\end{proof}

The last case to consider is of $UT_2$ with the action of $F(e_{11}+e_{22})\subseteq UT_2$ by left and right multiplication and we denote this case with $UT_2^F$. Note that, as observed in \cite[Section 5]{MR25a}, in this case we are dealing with ordinary polynomial identities. Therefore, the corresponding Hilbert series coincides with the one computed in \cite[Theorem 12]{BD12}, up to subtraction of the constant term (see Remark \ref{rem: constant term}). Of course, using the techniques of the previous theorems we obtain the same result.

\begin{theorem}
	The generalized Hilbert series of the relatively free algebra rank $k$ of $UT_2^F$ is \begin{gather*}
		\Hilb(UT_2^F; t_1,\dots,t_k) = -1 +\frac{2}{(1-t_1)\cdots (1-t_k)} + \frac{t_1+\ldots +t_k -1}{(1-t_1)^2\cdots (1-t_k)^2}
	\end{gather*}
\end{theorem}

\section{Asymptotics of generalized cocharacter multiplicities}

In this section we study the asymptotic behavior of the generalized cocharacters of a $W$-algebra $A$. At first, we show that it is related to the asymptotics of the ordinary cocharacters of $A$. Then we prove the analogues of several classical theorems, such as the hook theorem, the strip theorem and the bounds on the codimension sequences.

We obtain stronger results if we assume that the algebra $A$ has a unity. In this case, every unital subalgebra $B$ of $A$ induces a natural unital $W$-action, corresponding to $W=B$ acting by left and right multiplication. In fact, every $W$-action on $A$ is equivalent to an action of this kind, as we prove in the next Lemma.

 \begin{lemma}\label{lem: identity_good}
	Let $A$ be a $W$-algebra with unity $1_A$. Then, there exists an algebra homomorphism $\pi:W\rightarrow A$ such that
	\begin{equation}\label{eq: good_action}
		w a = \pi(w) a ,\qquad aw =a\pi(w)
	\end{equation}
	for each $w\in W$ and $a\in A$. Thus, the $W$-action $A$ is equivalent to a $\overline W$-action, where $\overline W=\img\pi\subseteq A$. 
	 \end{lemma}
 \begin{proof}
 		Let $1_A\in A$ be the unity of $A$ and consider the map $\pi:W\rightarrow A$ defined as $\pi(w)=w1_A$. Clearly, $\pi$ is linear. Moreover, for each $w_1,w_2\in W$, using the axioms \eqref{eq: W axioms} we have 
	    \[ \pi(w_1)\pi(w_2) = (w_1 1_A)(w_2 1_A) = w_1 (1_A (w_2 1_A)) = w_1(w_2 1_A)= (w_1w_2)1_A=\pi(w_1w_2)  \]
	    This shows that $\pi$ is an algebra homomorphism. Moreover, for each $w\in W$ and $a\in A$, we have \begin{gather*}
		        \pi(w)a=(w1_A)a=w(1_Aa)=wa  \\
		        a\pi(w) = a(w1_A)=(aw)1_A=aw
		\end{gather*}
	 \end{proof}

	\begin{remark}
	     The map $\pi$ in Lemma \ref{lem: identity_good} can be equivalently defined as $\pi(w)=1_A w$. In fact, for each $w\in W$, we have $w1_A=1_A(w1_A)=(1_Aw)1_A=1_Aw$.
	\end{remark}

    The map $\pi$ introduced in the Lemma above allows to define the evaluation of generalized polinomials with constant term. They can be formally defined as the elements of the algebra $\WX^+ = \WX\oplus W$, which is the free algebra in the category of unital $W$-algebras.
	
	In the following, we show that the generalized cocharacter multiplicities of an unital algebra $A$ can be bounded by a function depending on the ordinary cocharacter multiplicities. In the computations, we use the skew Shur polynomial so we recall their definition. 
    
	If $\lambda$ and $\mu$ are two partitions, we say that $\lambda\subseteq \mu $ if $\lambda_i\leq \mu_i$ for any $i=1,\dots$. In this case, the skew Shur polynomials are defined as 
	\[ s_{\mu\setminus\lambda} = \sum_{\nu} C^\mu_{\lambda\nu} s_\nu \] 
	where the $C^\nu_{\lambda\mu}$ are the Littlewood-Richardson coefficients.
	Note that $s_{\lambda}=s_{\lambda\setminus\emptyset}$. For an overview of the properties of skew Shur polynomials, the reader can see \cite{Mac98, Mel17}. We only need to recall their duplication formula:
	
	\begin{lemma}[\cite{Mac98}, Equation 5.10]\label{lem: duplication}
		Let $t_1,\dots,t_l$ and $v_1,\dots,v_k$ be two sets of disjoint variables. Then, for any partitions $\lambda\subseteq\mu$ we have
		\[ s^{l+k}_{\mu\setminus\lambda}(t_1,\dots,t_l,v_1,\dots,v_k) = \sum_{\lambda\subseteq \nu\subseteq \mu} s^l_{\mu\setminus\nu}(t_1,\dots,t_l) s^k_{\nu\setminus\lambda}(v_1,\dots,v_k) \]
	\end{lemma}
	
	The next lemma is the main tool used in the subsequent theorems and we believe it is also of independent interest.

\begin{lemma}\label{lem: multiplicity bound}
Let $A$ be a $\Polid$ $W$-algebra and let $d=\dim(W)$. For each $n\in\mathbb{N}$, consider its ordinary cocharacter sequence
$ \chi_n(A) = \sum_{\lambda \vdash n} a_\lambda \chi_\lambda $
and its generalized cocharacter sequence
$ g\chi_n(A) = \sum_{\lambda \vdash n} m_\lambda \chi_\lambda $. 
Then, for each $\lambda\vdash n$, we have \begin{equation}\label{eq: multiplicity bound}
    m_\lambda \leq \sum_{\substack{ \mu\vdash (2n+1) \\ \lambda \subseteq \mu}} a_\mu s^d_{\mu\setminus\lambda} (1,\dots,1)
\end{equation}
\end{lemma}
\begin{proof} 
    Let $w_1,\dots,w_d$ be a basis of $W$ and fix $k,n\in\bN$ such that $k\geq 2n+1$. Let $X_k=\{x_1,\dots,x_k\}$ and $Y_d=\{y_1,\dots,y_d\}$ be sets of variables and consider the free associative algebra $F\langle X_k\cup Y_d\rangle := F\langle x_1,\dots,x_k,y_1,\dots,y_d\rangle$ of rank $k+d$ and the free unital $W$-algebra $W\langle X_k\rangle^+:=W\langle x_1,\dots,x_k\rangle^+$ of rank $k$. There is a surjective algebra homomorphism 
    \begin{equation*}
    	\varphi:F\langle X_k\cup Y_d\rangle \rightarrow W\langle X_k\rangle^+
    \end{equation*}
    such that $\varphi(x_i) = x_i$ and $\varphi(y_j) = w_j$ for any $i=1,\dots,k$, $j=1,\dots,d$.

    If $A$ is a unital $W$-algebra and $\alpha: W\langle X_k\rangle^+\rightarrow A$ is an homomorphism of $W$-algebras, then there exists a unique associative algebra homomorphism $\bar \alpha: F\langle X_k\cup Y_d\rangle \rightarrow A$ such that the following diagram commutes.
    \[ \begin{tikzcd}
    	F\langle X_k\cup Y_d\rangle \arrow[r, "\varphi"] \arrow[dr, "\bar\alpha" ] & W\langle X_k\rangle^+ \arrow[d, "\alpha"] \\
    	{}  & A
    \end{tikzcd}\]
    The homomorphism $\bar\alpha$ is induced by $\bar\alpha(x_i) =\alpha(x_i)$ and $\bar\alpha(y_j) = \pi(w_j)$, for any $i=1,\dots,k$, $j=1,\dots,d$, where $\pi$ is the map obtained by Lemma \ref{lem: identity_good}.
    
    Therefore, if $f\in F\langle X_k\cup Y_d\rangle$ is an ordinary identity of $A$, then $\varphi(f)$ is a generalized identity of $A$. In fact, $\alpha(\varphi(f))=\bar\alpha(f)=0$ for any $W$-algebra homomorphism $\alpha: W\langle X_k\rangle^+\rightarrow A$. In other words, setting $\id(A)_{k,d}=\id(A)\cap F\langle X_k\cup Y_d\rangle$ and $\gid(A)_k=\gid(A)\cap W\langle X_k\rangle^+$, we have $\varphi(\id(A)_{k,d})\subseteq \gid(A)_k$.

    We now consider the following action of $GL_k(F)$ on $F\langle X_k\cup Y_d\rangle$. We let $u=(u_{ij})_{i,j\leq k}\in GL_k(F)$ act naturally on $X_d$, that is $u\ast x_j=\sum_ju_{ij}x_i$, and trivially on $Y_d$, i.e. $u *y_j=y_j$. Then, endowing $W\langle X_k\rangle^+$ with the usual $GL_k(F)$-action, $\varphi$ is a map of $GL_k(F)$-modules. Since it preserves identities, it also induces the homomorphism of $GL_k(F)$-modules 
    \[ \bar{\varphi}:\frac{F\langle X_k\cup Y_d\rangle}{\id(A)_{k,d}}\rightarrow \frac{W\langle X_k\rangle^+}{\gid(A)_k}=\F_k(A)^+.\]
    
    Let $U^{(2n+1)}_{k+d}\subseteq F\langle X_k\cup Y_d\rangle/\id(A)_{k,d}$ be the space of homogeneous polynomials of degree $2n+1$ and let $W^{(m)}_k\subseteq\F_k(A)$ be the space of homogeneous generalized polynomials of degree $m$. Then $\bar{\varphi}(U^{(2n+1)}_{k+d})= \bigoplus_{m=0}^n W^{(m)}_k$, so the multiplicity of a $GL_k(F)$-irreducible in $U^{(2n+1)}_{k+d}$ is greater than or equal to its multiplicity in $\bigoplus_{m=0}^n W^{(m)}_k$. By Theorem \ref{th: Hilbert}, the character of $\bigoplus_{m=0}^n W^{(m)}_k$ is \begin{equation}\label{eq: W_sum decomposition}
    	\sum_{m=0}^{n}\sum_{\lambda\vdash m} m_\lambda s^k_\lambda(t_1,\dots,t_k)
    \end{equation}
    On the other hand, the $GL_k(F)$-character of $U^{(2n+1)}_{k+d}$ can be obtained as a restriction of the ordinary homogeneous cocharacters. In fact, the action defined above on $F\langle X_k\cup Y_d\rangle$ is the restriction of the natural $GL_{k+d}(F)$-action when we consider the immersion $GL_k(F)\hookrightarrow GL_{k+d}(F)$ such that $u\mapsto \begin{pmatrix} u & 0 \\ 0 & I_d \end{pmatrix}$, where $I_d$ is the identity matrix of order $d$. Therefore, the character of $U^{(2n+1)}_{k+d}$ with respect to the $GL_k(F)$-action is the symmetric polynomial
    \[ p(t_1,\dots,t_k)=\sum_{\mu\vdash (2n+1)} a_\mu s^{k+d}_\mu (t_1,\dots,t_k,\underbrace{1,\dots,1}_{d\text{ times}}) \]
    Using the duplication formula of Lemma \ref{lem: duplication}, we can express $p(t_1,\dots,t_k)$ as \begin{equation} \label{eq: U decomposition}
    	p(t_1,\dots,t_k)= \sum_{\mu\vdash (2n+1)} \sum_{\lambda\subseteq\mu} a_\mu s^d_{\mu\setminus\lambda} (1,\dots,1) s^k_\lambda (t_1,\dots,t_k)
    \end{equation}
    Comparing the coefficients of $s^k_\lambda$ in the polynomials of Equations \eqref{eq: W_sum decomposition} and \eqref{eq: U decomposition}, we obtain the result.
\end{proof}

For any $k,l\geq0$, the hook $H(k,l)$ is the set of partitions $\lambda$ such that $\lambda_{k+1}\leq l$.
The celebrated Hook theorem \cite{AR82} states that the ordinary cocharacter multiplicities $\{a_\lambda\}_{\lambda\vdash n}$ of any PI-algebra are contained in some hook $H(k,l)$, that is $a_\lambda=0$ whenever $\lambda\notin H(k,l)$. In the next theorem we prove that the generalized cocharacter multiplicities of a unital $W$-algebra are contained in the same hook. Using this fact we are able to generalize several other bounds.

\begin{theorem}\label{th: good_hook}
	Let $A$ be a $\Polid$ $W$-algebra with unity whose ordinary cocharacter multiplicities are contained in the hook $H(k,l)$. Then, \begin{itemize}
		\item the generalized cocharacter multiplicities of $A$ are contained in $H(k,l)$,
		\item the generalized colength sequence $gl_n(A)$ is polynomially bounded in $n$,
		\item the generalized multilinear codimension sequence $gc_n(A)$ is exponentially bounded in $n$,
		\item the generalized homogeneous codimension sequences $gC_k^n(A)$ are polynomially bounded in $n$.
	\end{itemize}
\end{theorem}
\begin{proof}
    For each partition $\lambda$, denote with $a_\lambda$ the ordinary cocharacter multiplicities and with $m_\lambda$ the generalized cocharacter multiplicities. Then, $a_\mu=0$ for each $\mu\notin H(k,l)$. Note that, if the partition $\lambda$ is not in the hook $H(k,l)$, then $\mu\notin H(k,l)$ for each $\mu\supseteq \lambda$. So, by Lemma \ref{lem: multiplicity bound}, we get that $m_\lambda=0$ for each $\lambda\notin H(k,l)$. Therefore, the generalized cocharacter multiplicities are contained in $H(k,l)$.
    
    To prove the other bounds we need to bound the multiplicities $m_\lambda$, for each $\lambda\in H(k,l)$. The dimension of the Weyl module $\mathcal{W}_\lambda^d$ is given by \cite[Appendix A]{FH91}: \begin{equation}\label{eq: weyl dimension}
		\dim \mathcal{W}^d_\lambda = s_{\lambda}^d(1,\dots,1)= \prod_{i<j\leq k}\frac{\lambda_i-\lambda_j+j-i}{j-i}
	\end{equation}
    The previous formula readily gives that $s_\lambda^d(1,\dots,1)=\dim W^d_\lambda$ is polynomially bounded in $|\lambda|$. By \cite[Lemma 11]{Ber96}, whenever $\lambda\in H(k,l)$, the Littlewood-Richardson coefficient $C^\lambda_{\mu\nu}$ is bounded by a polynomial in $|\lambda|$, depending only on $k$ and $l$. Therefore, if $\mu \in H(k,l)$ and $\lambda\subseteq \mu$, then $s^d_{\mu\setminus\lambda}(1,\dots,1) = \sum_\nu C^\mu_{\nu\lambda} s^d_\nu (1,\dots,1)$ is polynomially bounded in $|\mu|$. Therefore, for any $\lambda\vdash n$, applying the aforementioned bounds to the right hand side of Equation \eqref{eq: multiplicity bound}, we have that there exist some constants $C,t>0$ such that 
    \[ m_\lambda \leq C(2n+1)^t \sum_{\mu\vdash 2n+1} a_\mu \]
    Now, $l_{2n+1}(A)=\sum_{\mu\vdash 2n+1} a_\mu$ is the ordinary colength and, by \cite{BR83}, is polynomially bounded in $n$. Therefore, for each $\lambda\in H(k,l)$, the multiplicities $m_\lambda$ are polynomially bounded in $n=|\lambda|$, that is, there exist some constants $C,t>0$, depending on $k$ and $l$, such that 
    \[ m_\lambda \leq Cn^t \]
    Since the number of different partitions $\lambda\in H(k,l)$ is polynomially bounded in $|\lambda|$, then the generalized colength sequence 
    \[ gl_n(A) = \sum_{\lambda\vdash n} m_\lambda = \sum_{\substack{\lambda\vdash n \\ \lambda\in H(k,l)}} m_\lambda 
    \]
    is polynomially bounded in $n$.
    
    Similarly, the bound on the generalized multilinear codimensions follow from Equations \eqref{eq: gen_cod_formula} using the fact that, if $\lambda\in H(k,l)$, $d_\lambda=\dim\cS_\lambda$ is exponentially bounded in $|\lambda|$ by \cite{BR87}. The bound on the generalized homogeneous codimensions follow from \eqref{eq: gen_hom_cod_formula} together with the fact that $\dim \mathcal{W}^d_\lambda$ is polynomially bounded in $|\lambda|$ by Equation \eqref{eq: weyl dimension}.    
\end{proof}

The strip of height $k$ is the set of partitions of height at most $k$ and coincides with the hook $H(k,0)$. The strip theorem \cite{Reg79} states that the ordinary cocharacter multiplicities of an algebra $A$ are contained in a strip if and only if $A$ satisfies some Capelli identity. As a corollary of Theorem \ref{th: good_hook}, we can immediately obtain a generalization of the strip theorem for unitary $W$-algebras. Recall that the Capelli polynomial of rank $m$ is defined as
\[
Cap_m(x_1,\ldots,x_m;y_1,\ldots,y_{m+1})=\sum_{\sigma \in S_n} (-1)^\sigma y_1x_{\sigma(1)}y_2x_{\sigma(2)}\cdots y_mx_{\sigma(m)}y_{m+1}
\]
Note that, if $A$ is a finite dimensional $W$-algebra with $\dim_F A=n$, then $A$ satisfies the Capelli identity of rank $n+1$. 

\begin{corollary}\label{cor: hook_unity}
	Let $A$ be a $W$-algebra with unity. The generalized cocharacters of $A$ are contained in a strip of height $k$ if and only if $A$ satisfies the Capelli identity $Cap_{k+1}$ of rank $k+1$.
\end{corollary}
\begin{proof}
	By Equation \eqref{eq: multiplicity bound}, the generalized cocharacters are in $H(k,0)$ if and only if the ordinary cocharacters are in $H(k,0)$, and this is equivalent to $Cap_{k+1}\in \id(A)$, by the strip theorem \cite{Reg79}.
\end{proof}

Now, we turn our attention to arbitrary $W$-algebras, possibly without unity. To deal with the generic case we need to introduce the semidirect product.

\begin{definition}
	Let $A$ be a $W$-algebra. The semidirect product between $W$ and $A$, denoted by $\wa$, is the algebra defined on the direct sum of vector spaces $A\oplus W$ with multiplication given by
	\[ (a_1,w_1)(a_2,w_2) = (a_1a_2 + w_1a_2 +a_1w_2, w_1w_2) \]
    It is naturally a $W$-algebra with respect to left and right multiplication.
\end{definition}

The semidirect product $\wa$ is naturally a $W$-algebra with action given by
\[ (a,w_1)w_2=(aw_2,w_1w_2),\qquad w_2(a,w_1)=(w_2a,w_2w_1)  \]
for any $a\in A$, $w_1,w_2\in W$. Moreover, we can identify both $A$ and $W$ with suitable subalgebras of $\wa$ in a natural way and, with this identification, $A$ is an ideal. This can be formally expressed trough the following split short exact sequence of algebras: 
\[ \begin{tikzcd}
    0 \arrow[r] & A \arrow[r] & \wa \arrow[r,shift left] & W \arrow[l, shift left] \arrow[r]  & 0
\end{tikzcd} \]
The semidirect product $\wa$ has a unity, corresponding to $(0,1_W)$. Note that the algebra of generalized polynomials with constant term is precisely $\WX^+=\WX\rtimes W$.

\begin{lemma}\label{lem: wa is PI}
	Let $A$ be a $W$-algebra. Then $\wa$ is $\Polid$ if and only if $A$ is $\Polid$, and in this case $\id(\wa)\subseteq \id(A)$. 
	
	Similarly, $\wa$ is $\gpi$ if and only if $A$ is $\gpi$ and $\gid(\wa)\subseteq \gid(A)$.
\end{lemma}
\begin{proof}
	Clearly, $\id(\wa)\subseteq \id(A)$ (resp. $\gid(\wa)\subseteq \gid(A)$) since $A$ is isomorphic to a subalgebra of $\wa$. We need to prove that when $A$ is $\Polid$ (resp. $\gpi$), then $\wa$ has the same property. Let $f(x_1,\dots,x_n)$ be a non trivial ordinary (resp. generalized) polynomial identity of $A$. Recall that $W$ is finite dimensional, say $\dim(W)=d$, so the Capelli polynomial $Cap_{d+1}$ is an identity of $W$. Therefore, every evaluation of $Cap_{d+1}$ in $\wa$ has image in $A$. We conclude that the composition $f(Cap_{d+1},\dots,Cap_{d+1})$ is a polynomial identity of $\wa$.
\end{proof}

In the next theorem we generalize the hook theorem to arbitrary $W$-algebras. Note that in this case we do not have any explicit relationship between the hook given by ordinary cocharacters and the one given by the generalized cocharacters. We also obtain the generalization of the other bounds and show that they are equivalent to the existence of an ordinary identity. This theorem is the direct analogue of \cite[Theorems 14, 29]{Ber96} in the context of generalized identities.

\begin{theorem}\label{th: hook}
	Let $A$ be a $W$-algebra. The following are equivalent: \begin{enumerate}
		\item $A$ satisfies an ordinary polynomial identity, \label{it: 0}
		\item the generalized cocharacter multiplicities of $A$ are contained in a hook, \label{it: 1}
		\item the generalized colength sequence of $A$ is polynomially bounded, \label{it: 2}
		\item the generalized multilinear codimension sequence of $A$ is exponentially bounded, \label{it: 3}
		\item the generalized homogeneous codimension sequences of $A$ are polynomially bounded. \label{it: 4}
	\end{enumerate}
\end{theorem}
\begin{proof}
	Assume that $A$ satisfies an ordinary polynomial identity. Then, by Lemma \ref{lem: wa is PI}, the semidirect product $\wa$ is $\Polid$, so it satisfies the hypothesis of Theorem \ref{th: good_hook}. In particular, the generalized cocharacter multiplicities of $\wa$ are contained in a hook. On the other hand, $\gid(\wa)\subseteq \gid(A)$ implies that the generalized cocharacter multiplicities of $A$ are smaller than those of $\wa$, therefore they are contained in a hook, too. Similarly, one can prove the other bounds.
	
    On the converse, if $A$ does not satisfy any ordinary polynomial identity, then the free algebra $F\langle X\rangle$ is contained in the relatively free algebra $\F(A)=\WX /\gid(A)$. But $F\langle X\rangle$ does not satisfy any of the conditions \ref{it: 1}, \ref{it: 2}, \ref{it: 3} or \ref{it: 4}, so $A$ does not satisfy them, too.
\end{proof}

To obtain a generalization of the strip theorem for arbitrary $W$-algebras we will use a direct proof involving the representation theory of the symmetric group. We recall the branching theorem for the irreducible representations of $S_n$. 

\begin{theorem}[Theorem 2.3.1 of \cite{GZ05}] \label{th: branching} 
    Consider the natural embedding of $S_n$ in $S_{n+1}$ as the subgroup fixing $n+1$. Then, \begin{itemize}
    \item if $\lambda\vdash n$, then $\cS_\lambda^{\uparrow S_{n+1}}\cong \sum_{\mu\in\lambda^+}\cS_\mu$ where $\lambda^+$ is the set of all partitions of $n+1$ whose Young diagram is obtained by adding one box to the Young diagram of $\lambda$;
    \item if $\lambda\vdash n+1$, then $\cS_\lambda^{\downarrow S_{n}}\cong \sum_{\mu\in\lambda^-}\cS_\mu$ where $\lambda^-$ is the set of all partitions of $n$ whose Young diagram is obtained by removing one box to the Young diagram of $\lambda$.
    \end{itemize}
\end{theorem}

We also need to introduce the following set of polynomials:

\begin{definition}
	The generalized Capelli set of rank $m$ in the free $W$-algebra $\WX$ is the set of all the polynomials obtained from $ Cap_m(x_1,\ldots,x_m; y_1,\ldots,y_{m+1})$ by eventually setting the variables $y_i$ equal to some element $w\in W$ in all possible ways.

    The $W$-algebra $A$ satisfies the generalized Capelli set of rank $m$ if every polynomial of this set is a generalized identity of $A$.
\end{definition}

A simple observation is that every generalized polynomial alternating in $m$ variables is a consequence of the generalized Capelli set of rank $m$ (see \cite[Proposition 1.5.5]{GZ05} for the proof in the ordinary case).
Note that if $A$ has a unity, then it satisfies the $m$-th Capelli set if and only if $Cap_m$ is a polynomial identity. So the next theorem generalizes Corollary \ref{cor: hook_unity}.

\begin{theorem} \label{Striptheorem}
    Let $A$ be a $W$-algebra. The generalized cocharacters of $A$ are contained in a strip of height $k$ if and only if $A$ satisfies the generalized Capelli set of rank $k+1$.
\end{theorem}
\begin{proof}
    Assume that $A$ satisfies the Capelli set of rank $k+1$ and fix an irreducible module $M\subseteq gP_n(A)$ such that $M\cong\cS_\lambda$ with $h(\lambda) \geq k+1$. Let $M^{\downarrow S_{k+1}}$ be the restriction  of $M$ to the symmetric group $S_{k+1}$ acting on the first $k+1$ variables. By iteratively applying the Branching Theorem \ref{th: branching}, we see that $M^{\downarrow S_{k+1}}$ contains, with some non zero multiplicity, an $S_{k+1}$-irreducible submodule corresponding to a column of height $k+1$, which is precisely the alternating representation of $S_{k+1}$. In other words, there exists a non zero polynomial $f\in M$ alternating in a set of $k+1$ variables. Therefore, $f$ is a consequence of the generalized Capelli set of rank $k+1$ and, by the irreducibility of $M$, we have $M\subseteq\gid(A)$. Therefore we proved that $m_\lambda=0$ for each partition $\lambda$ of height greater than or equal to $k$.

    Conversely, assume that the generalized cocharacters sequence of $A$ lies in the strip of height $k$. Let $f$ be a polynomial in the Capelli set of rank $k+1$ and assume that $f$ has $n$ variables. Clearly, $k+1\leq n\leq 2k+3$ and $f$ is alternating in the set of $k+1$ variables $\{x_1,\dots,x_{k+1}\}$. Therefore, $F[S_{k+1}]f$ is isomorphic to the alternating representation of $S_{k+1}$. By iteratively applying the branching theorem, all the irreducible submodules of $F[S_n]f\cong(F[S_{k+1}]f)^{\uparrow S_n}$ with non zero multiplicity have height at least $k+1$. So $F[S_n]f \subseteq \gid(A)$ and $f$ is a generalized polynomial identity of $A$.
\end{proof}

If $A$ is a finite dimensional $W$-algebra with $\dim_F A=n$, then $A$ satisfies the Capelli identity of rank $n+1$. Therefore, by Theorem $\ref{Striptheorem}$, we get 
\begin{corollary}
	Let $A$ be a finite dimensional $W$-algebra, $\dim_F A=k$. Then for any $n \geq 1$, 
	\[
	\chi_n(A)=\sum_{\substack{\lambda \vdash n \\ h(\lambda) \leq k}}m_\lambda \chi_\lambda,
	\]
	i.e., the $n$-th generalized cocharacter of $A$ lies in a strip of height $k$.
\end{corollary}

\section{Generators of varieties of $W$-algebras}

This section is devoted to the proof of some important corollaries of the strip theorem and of the hook theorem. The celebrated Kemer's representability theorem \cite{Kem84} states that every ordinary PI-algebra is PI-equivalent to the Grassmann envelope of a finite dimensional superalgebra. Moreover, if the algebra satisfies a Capelli's identity, then it is PI-equivalent to a finite dimensional algebra. This result has been extended to algebras with an $H$-module action of some finite dimensional semisimple Hopf algebra $H$ \cite{Kar16}. The first step in the proof of the representability theorems is the reduction of the problem to the case of finitely generated algebras. More precisely, one can prove that every variety of PI-algebras is generated by the Grassmann envelope of a finitely generated algebra and this is a corollary of the hook theorem. We wish to extend this result to the setting of $W$-algebras and generalized identities.
To reach this goal we need to introduce $W$-superalgebras and their Grassmann envelopes. Recall that an associative superalgebra (or $\mathbb{Z}_2$-graded algebra) is an associative algebra $A$ with a vector space decomposition $A=A_0\oplus A_1$ such that $A_0^2+A_1^2\subseteq A_0$ and $A_0A_1+A_1A_0\subseteq A_1$. 

\begin{definition}
	A $W$-superalgebra is a superalgebra $A=A_0 \oplus A_1$ endowed with a $W$-action such that $A_0W + WA_0\subseteq A_0$ and $A_1W + WA_1\subseteq A_1$. That is, if $a\in A$ is an homogeneous element, then $\deg(aw)=\deg(wa)=\deg(a)$ for each $w\in W$.
\end{definition}

Fix two countable sets of variables $Y=\{y_1,y_2,\dots\}$ and $Z=\{z_1,z_2,\dots\}$. The free $W$-superalgebra is defined as the algebra $\WYZ$ endowed with the natural $W$-action and the grading induced by $\deg(y_i)=0$ and $\deg(z_i)=1$ for each $i=1,2,\dots$. The polynomial $f\in\WYZ$ is a generalized graded identity of the $W$-superalgebra $A$ if it vanishes under every admissible evaluation in $A$, that is when we evaluate every even variable $y_i$ in $A_0$ and every odd variable $z_i$ in $A_1$. We denote the $T_W$-ideal of graded identities of $A$ as $\gid_2(A)$. This allows to develop a gPI-theory for $W$-superalgebras in analogy with the ungraded case. Therefore, we can introduce the notions of variety of $W$-superalgebras, of multilinear graded generalized polynomials, of graded generalized codimensions sequence, etc. Let's also define the space $gP_{l,m}\subseteq \WYZ$, for each $l,m\geq 0$, as the space of multilinear graded generalized polynomials in the variables 
$y_1,\dots,y_l$ and $z_1,\dots,z_m$.
 
Let $A=A_0 \oplus A_1$ be a superalgebra. Let us denote with \[E(A)=(A_0 \otimes E_0) \oplus (A_1 \otimes E_1)\] the Grassmann envelope of $A$, where $E=E_0\oplus E_1$ is the Grassmann algebra. If $A$ is a $W$-superalgebra then $E(A)$ is also a $W$-superalgebra with the $W$-action given by:
\[
w*(a\otimes g)=(wa)\otimes g, \qquad (a\otimes g)*w=(aw)\otimes g,
\]
for all $w \in W, a \in A_0\cup A_1, g \in E_0\cup E_1$. 

The graded generalized identities of $A$ and of $E(A)$ are related by the linear isomorphism $\sim:gP_{l,m} \to gP_{l,m}$ defined as follows. Let $f \in gP_{l,m}$ and write $f$ as
\[
f=\sum_{\substack{\sigma \in S_m \\ V=(v_0,v_1,\ldots,v_m)}} \alpha_{\sigma,V}\ v_0z_{\sigma(1)}v_1\cdots v_{m-1}z_{\sigma(m)}v_m
\]
where $v_0,v_1,\ldots,v_m$ are generalized monomials in the variables $y_1,\ldots,y_l$ and $\alpha_{\sigma,V} \in F$. Then, $\tilde{f}$ is defined as
\[
\tilde{f}= \sum_{\substack{\sigma \in S_m \\ V=(v_0,v_1,\ldots,v_m)}} (-1)^\sigma \alpha_{\sigma,V}\ v_0z_{\sigma(1)}v_1\cdots v_{m-1}z_{\sigma(m)}v_m.
\]
As in the ordinary case, we can prove the following:
\begin{lemma}\label{3.7.4}
Let $f\in gP_{l,m}$. Then: 
  \begin{enumerate}
    \item [1)] $\tilde{\tilde{f}}=f$;
    \item [2)] $f$ is a graded generalized identity of $E(A)$ if and only if $\Tilde{f}$ is a graded generalized identity of $A$.
  \end{enumerate}
\end{lemma}

We also need the following technical result.

\begin{lemma}
    \label{2.5.6}
    Let $\lambda \vdash n, T_\lambda$ a Young tableau of shape $\lambda$ and $f=e_{T_\lambda}g$ where $g=g(x_1,\ldots,x_n)$ is some multilinear generalized polynomial in $x_1,\ldots,x_n$. If $\lambda \in H(k,l)$, then there exists a decomposition of $X^{(n)}=\{ x_1,\ldots,x_n\}$ into a disjoint union
    \begin{equation}
        \label{2.4}
        X^{(n)}=X_1 \cup \ldots \cup X_{k'}\cup Y_1 \cup \ldots \cup Y_{l'}
    \end{equation}
    with $k' \le k, l' \le l$ and a multilinear polynomial $f'=f'(x_1,\ldots,x_n)$ such that 
    \begin{itemize} 
        \item $f'$ is symmetric on the variables in each set $X_i, 1 \le i \le k'$;
        \item $f'$ is alternating on the variables in each set $Y_j, 1 \le j \le l'$;
        \item $FS_nf=FS_nf'$;
        \item the integers $k',l', |X_1|,\ldots,|X_{k'}|,|Y_1|,\ldots,|Y_{l'}|$ are uniquely determined by $\lambda$ and do not depend on the choice of the tableau $T_\lambda$;
        \item the decomposition \eqref{2.4} is uniquely defined by $T_\lambda$ and does not depend on $g$.
    \end{itemize}
\end{lemma}
\begin{proof}
	The proof of this lemma can be adapted from \cite[Lemma 2.5.6]{GZ05}. In fact, in that Lemma, it is proved that there exist a partition 
	\[X^{(n)}=X_1 \cup \ldots \cup X_{k'}\cup Y_1 \cup \ldots \cup Y_{l'}\]
	and an element $ue_{T_\lambda} \in FS_n e_{T_\lambda}$ such that \begin{itemize}
		\item for any fixed $X_i$, $i=1,\dots,k'$, and any permutation $\sigma$ of the elements in $X_i$ we have $\sigma ue_{T_\lambda} =ue_{T_\lambda}$,
		\item for any fixed $Y_i$, $i=1,\dots,l'$, and any permutation $\sigma$ of the elements in $Y_i$ we have $\sigma ue_{T_\lambda} = (-1)^\sigma ue_{T_\lambda} $
	\end{itemize} 
	Moreover, the partition of $X^{(n)}$ depends only on $T_\lambda$ and $k'$, $l'$ and the cardinalities of the partition classes depend only on $\lambda$. Clearly, $f'= ue_{T_\lambda}g$ is the desired element.
\end{proof}

If $\cV$ is a variety of $W$-algebras, we denote with $\cV^*$ the variety of $W$-superalgebras such that $A\in\cV^*$ if and only if $E(A)\in \cV$.

\begin{theorem}
\label{th: Grassmann envelope}
    Let $W$ be a finite dimensional associative unital algebra and let $A$ be a $W$-algebra whose generalized cocharacter multiplicities are contained in the hook $H(k,l)$. Then, there exists a finitely generated superalgebra $L=L^{(0)} \oplus L^{(1)}$ with $k$ even generators and $l$ odd generators such that $\gid(A)=\gid(E(L))$.
\end{theorem}
\begin{proof}
    Let $\cV=\var^W(A)$ and consider the variety of $W$-superalgebras $\mathcal{V}^*$ and its relatively free superalgebra 
    \[ L = \F_{k,l}(\cV^*)=\frac{W\langle y_1,\dots,y_k,z_1,\dots,z_l\rangle}{W\langle y_1,\dots,y_k,z_1,\dots,z_l\rangle\cap \gid_2(\cV^*)} \]
    on $k$ even free generators $l$ odd free generators. We shall prove that $\mathcal{V}$ is generated by $E(L)$.

    Clearly, $E(L) \in \mathcal{V}$ by the definition of $\mathcal{V}^*$. We next verify that $A$ satisfies all identities of $E(L)$.

    Let $f \in \gid(E(L))$ be a multilinear identity of $G(L)$ of degree $n$. We can assume that $f=e_{T_\lambda}g$ for some multilinear polynomial $g\in gP_n$. If $\lambda \notin H(k,l)$, then $f$ is an identity of $A$ since $g\chi_n(A)$ lies in the hook $H(k,l)$. Now, let $\lambda \in H(k,l)$. By applying Lemma \ref{2.5.6}, we may assume that
    \[
    f=f(y_1^1,\ldots,y_1^{m_1},\ldots,y_k^1,\ldots,y_k^{m_k},z_1^1,\ldots,z_1^{r_1},\ldots,z_l^1,\ldots,z_l^{r_l})
    \]
    depends on $k$ (possibly empty) sets of symmetric variables $\{ y_i^1,\ldots,y_i^{m_i} \}, 1 \le i \le k$ and on $l$ (possibly empty) sets of alternating variables $\{z_j^1,\ldots,z_j^{r_j}\}, 1 \le j \le l$. If we consider all $y_i^j$ as even variables and all $z_i^j$ as odd variables, then $f$ can be considered as a graded identity of the superalgebra $G(L)$. So, by Lemma \ref{3.7.4}, the superalgebra $L$ satisfies the graded identity $\Tilde{f}$. Now, we substitute all the variables in the set $\{ y_i^1,\ldots,y_i^{m_i} \}$ with the even generator $y_i\in L$, $1 \le i \le k$, and all the variables in the set $\{z_j^1,\ldots,z_j^{r_j}\}$ with the odd generator $z_j\in L$, $1 \le j \le l$. We obtain that \begin{equation} \label{eq: f=0 in L}
        \Tilde{f}(y_1,\ldots,y_1,\ldots,y_k,\ldots,y_k,z_1,\ldots,z_1,\ldots z_l,\ldots, z_l) = 0
    \end{equation}
    in $L$.

    Consider the superalgebra $A\otimes E = (A\otimes E_0) \oplus (A\otimes E_1)$ where the grading is determined by $E$. We note that $E(A\otimes E)= A\otimes R$ where $R=(E_0\otimes E_0) \oplus (E_1 \otimes E_1)$ is a commutative algebra, so $E(A\otimes E)$ satisfies all multilinear identities of $A$. Therefore, $E(A\otimes E)\in \cV$ and, by the definition of $\cV^*$, $A\otimes E\in \cV^*$. Now take some arbitrary elements $\Bar{y}_1^1,\ldots,\Bar{y}_1^{m_1}, \ldots,\Bar{y}_k^1,\ldots,\Bar{y}_k^{m_k}, \Bar{z}_1^1,\ldots,\Bar{z}_1^{r_1}, \ldots, \Bar{z}_l^1,\ldots,\Bar{z}_l^{r_l}$ in $A$, and consider 
    \begin{align*}
        q_i=\Bar{y}_i^1 \otimes a_i^1+\cdots+\Bar{y}_i^{m_i}\otimes a_i^{m_i}, \quad i=1,\ldots,k\\
        p_j=\Bar{z}_j^1\otimes b_j^1+\cdots+\Bar{z}_j^{r_j} \otimes b_j^{r_j}, \quad j=1,\ldots,l
    \end{align*}
    where $a_1^1,\ldots,a_k^{m_k} \in E_0$, $b_1^1,\ldots,b_l^{r_l}\in E_1$ are monomials of the Grassmann algebra $E$ written on distinct generators of $E$, hence $a_1^1\cdots a_k^{m_k}b_1^1\cdots b_l^{r_l} \neq 0$.

    Since $L$ is a relatively free algebra of $\cV^*$ and $A\otimes E\in \cV^*$, there exists an homomorphism of $W$-superalgebras sending $y_i\in L$ in $q_i\in A\otimes E$, $1 \le i \le k$, and $z_j\in L$ in $p_j$, $1 \le j \le l$. Applying such homomorphism to Equation \ref{eq: f=0 in L}, we get
    \begin{align*}
    0=\Tilde{f}(q_1,\ldots,p_l)=&\\
    m_1!\cdots m_k!r_1!&\cdots r_l!\ f(\Bar{y}_1^1,\ldots,\Bar{y}_1^{m_1},\ldots,\Bar{z}_l^1,\ldots,\Bar{z}_l^{r_l}) \otimes a_1^1\cdots a_1^{m_1}\cdots b_l^1 \cdots b_l^{r_l},
 \end{align*}
 since $(a_i^j)^2=(b_i^j)^2=0$ for all $i,j$, and $\Tilde{f}$ is symmetric on any set $\{ y_i^1,\ldots,y_i^{m_i} \}$ and on any set $\{ z_j^1,\ldots,z_j^{r_j} \}$.
 Therefore, $f$ vanishes on every evaluation of the variables with elements of $A$, that is $f\in \gid(A)$. We conclude that $\gid(E(L))\subseteq \gid(A)$ and we are done.
\end{proof}

An important corollary of the previous theorem is the following.

\begin{corollary}\label{cor: fin_gen}
	Let $\cV$ be a variety of $W$-algebras. If $\cV$ satisfies a generalized Capelli set, then $\cV=\cV(B)$ for some finitely generated $W$-algebra $B$.
\end{corollary}
\begin{proof}
    By \ref{Striptheorem}, the cocharacter sequence of $\cV$ is contained in a strip, that is a hook $H(k,0)$ for some integer $k$. By Theorem \ref{th: Grassmann envelope} applied to the relatively free $W$-algebra $\F(\cV)=\WX / \gid(\cV)$, there exists a superalgebra $L=L_0\oplus L_1$ generated by $k$ even elements such that $\cV=\var^W(E(L))$. Since all generators of $L$ are even, we have $L_1=0$. Therefore, $E(L)=L\otimes E_0$ and, since $E_0$ is a commutative algebra, $\gid(E(L))=\gid(L)$. So, $\cV=\var^W(L)$.
\end{proof}

In light of the previous results, in order to obtain a generalized version of the representability theorem it is enough to prove the following conjecture:

\begin{conjecture}
    Let $A$ be a finitely generated $W$-(super)algebra. Then, there exists a finite dimensional $W$-(super)algebra $B$ such that $\gid (A)=\gid(B)$.
\end{conjecture}

\section{Declarations}
The authors report there are no competing interests to declare.

\bibliographystyle{amsalpha}

\end{document}